\documentclass[11pt]{article}
\usepackage{amscd}
\usepackage{amsmath}
\usepackage{amsfonts}
\usepackage[cm]{fullpage}
\usepackage[small]{titlesec}
\newcommand{\de}{\partial}
\newcommand{\db}{\overline{\partial}}
\newcommand{\mn}{\sqrt{-1}}

\newcommand{\ve}{\varepsilon}
\newcommand{\vp}{\varphi}

\newcommand{\Ric}{\mathrm{Ric}}
\newcommand{\ov}[1]{\overline{#1}}

\newcommand{\tr}[2]{\textrm{tr}_{#1} #2}

\begin{document}
\newcounter{remark}
\newcounter{theor}
\setcounter{remark}{0}
\setcounter{theor}{1}
\newtheorem{claim}{Claim}
\newtheorem{theorem}{Theorem}[section]
\newtheorem{proposition}{Proposition}[section]
\newtheorem{conjecture}{Conjecture}[section]
\newtheorem{question}{Question}[section]
\newtheorem{lemma}{Lemma}[section]
\newtheorem{defn}{Definition}[theor]
\newtheorem{corollary}{Corollary}[section]
\newenvironment{proof}[1][Proof.]{\begin{trivlist}
\item[\hskip \labelsep {\itshape #1}]}{\hfill$\square$\medskip\end{trivlist}}
\newenvironment{remark}[1][Remark]{\addtocounter{remark}{1} \begin{trivlist}
\item[\hskip
\labelsep {\bfseries #1  \thesection.\theremark}]}{\end{trivlist}}
\setlength{\arraycolsep}{2pt}
\centerline{\bf \large The complex Monge-Amp\`ere equation}
\smallskip

\centerline{\bf \large on compact Hermitian manifolds\footnote{Research supported in part by National Science Foundation grant DMS-08-48193.  The second-named author is also supported in part by a Sloan Foundation fellowship.}}

\bigskip
\bigskip
\centerline{\bf Valentino Tosatti and Ben Weinkove}

\bigskip

\begin{abstract}  We show that, up to scaling, the complex Monge-Amp\`ere equation on compact Hermitian manifolds always admits a smooth solution.
\end{abstract}

\bigskip

\section{Introduction} \label{intro}

Let $(M, g)$ be a compact Hermitian manifold of complex dimension $n \ge 2$ and write $\omega$ for the corresponding real $(1,1)$ form
$$\omega = \sqrt{-1} \sum_{i,j} g_{i \ov{j}} dz^i \wedge d\ov{z^j}.$$  For a smooth real-valued function $F$ on $M$, consider the complex Monge-Amp\`ere equation 
\begin{equation} \begin{split} \label{ma}
&(\omega + \sqrt{-1} \partial \ov{\partial} \varphi )^n   = e^F \omega^n, \ \ \textrm{with} \\
 & \omega + \sqrt{-1} \partial \ov{\partial} \varphi>0 , \quad \sup_M \varphi =0,
\end{split} \end{equation}
for a real-valued function $\varphi$.

Our main result is as follows.

\bigskip
\noindent
{\bf Main Theorem} \,  \emph{Let $\varphi$ be a smooth solution of the complex Monge-Amp\`ere equation (\ref{ma}).  Then there are uniform $C^{\infty}$ a priori estimates on $\varphi$ depending only on $(M, \omega)$ and $F$.}
\bigskip

A corollary of this is that we can solve (\ref{ma}) uniquely after adding a constant to $F$, or equivalently, up to scaling the volume form $e^F \omega^n$.

\bigskip
\noindent
{\bf Corollary 1.} \, \emph{For every smooth real-valued function $F$ on $M$ there exists a unique real number $b$ and a unique smooth real-valued function $\varphi$ on $M$ solving
\begin{equation} \begin{split} \label{ma2}
&(\omega + \sqrt{-1} \partial \ov{\partial} \varphi )^n   = e^{F+b} \omega^n, \ \ \textrm{with} \\
 & \omega + \sqrt{-1} \partial \ov{\partial} \varphi>0 , \quad \sup_M \varphi =0.
\end{split} \end{equation}
}

In the  case of $\omega$ K\"ahler, that is when $d \omega=0$, this result is precisely the celebrated Calabi Conjecture \cite{C} proved by Yau \cite{Y}. We note here that if $\omega$ satisfies 
\begin{equation} \label{gl}
\de\db\omega^k=0, \quad \textrm{for } k=1,2,
\end{equation}
 (in particular if $\omega$ is closed) then the constant $b$ must equal
$$\log \frac{\int_M\omega^n}{\int_M e^F\omega^n}.$$

We mention now some special cases where the results of the Main Theorem and Corollary are already known.  Cherrier \cite{ChP} gave a proof when the complex dimension is two or if $\omega$ is balanced, that is, $d(\omega^{n-1})=0$ (an alternative proof was very recently given in \cite{TW}).   In addition, Cherrier \cite{ChP} dealt with the case of conformally K\"ahler and considered a technical assumption which is slightly weaker than balanced, 
see also the related work of Hanani \cite{Ha}.   Guan-Li \cite{GL} gave a proof under the assumption (\ref{gl}).  For further background we refer the reader to \cite{TW} and the references therein.

As the reader will see in the proof below, we note that  the key $L^{\infty}$ bound of $\varphi$ in the Main Theorem follows from combining a lemma of \cite{ChP} with  some recent estimates  of the authors \cite{TW}.

Finally, we remark that one can give a geometric interpretation of \eqref{ma2} in terms of the first Chern class $c_1(M)$ of $M$. We denote by $\Ric(\omega)$ the first Chern form of the Chern connection of $\omega$, which is a closed form cohomologous to $c_1(M)$. We then consider the real Bott-Chern space $H^{1,1}_{\mathrm{BC}}(X,\mathbb{R})$ of closed real $(1,1)$ forms modulo the image of $\mn\de\db$ acting on real functions. It has a natural surjection to the familiar space $H^{1,1}(M,\mathbb{R})$, which is an isomorphism if and only if $b_1(M)=2h^{0,1}$ \cite{Ga2} (in particular if $M$ is K\"ahler). The form $\Ric(\omega)$ determines a class $c_1^{\mathrm{BC}}(M)$ in  $H^{1,1}_{\mathrm{BC}}(M,\mathbb{R})$ which maps to the usual first Chern class $c_1(M)$ via the above surjection. Then from our main theorem we get the following Hermitian version of the Calabi conjecture (see also a related question of Gauduchon \cite[IV.5]{Ga2}):

\bigskip
\noindent
{\bf Corollary 2.} \, \emph{Every representative of the first Bott-Chern class $c_1^{\mathrm{BC}}(M)$ can be represented as the first Chern form of a Hermitian metric of the form $\omega+\mn\de\db\vp$.}\\

To see why this holds, just notice that \eqref{ma2} holds for some constant $b$ if and only if
\begin{equation}\label{ma3}
\Ric(\omega+\mn\de\db\vp)=\Ric(\omega)-\frac{\mn}{2\pi}\de\db F,
\end{equation}
and that by definition every form representing $c_1^{\mathrm{BC}}(M)$
can be written as $\Ric(\omega)-\frac{\mn}{2\pi}\de\db F$ for some function $F$.  We note here that in the case $n=2$ \cite[Corollary 2]{TW} gives a criterion to decide which representatives of $c_1(M)$ can be written in this form.
\bigskip

\setcounter{equation}{0}
\section{Proof of the Main Theorem}

By the results of \cite{ChP}, \cite{GL}, \cite{Zh} it suffices to obtain a uniform bound of $\varphi$ in the $L^{\infty}$ norm.    Indeed, by extending the second order estimate on $\varphi$ of Yau \cite{Y} (and Aubin \cite{A}),
Cherrier \cite{ChP} has shown, for general $\omega$, that a uniform $L^\infty$ bound on $\vp$ implies that the metric $\omega+\mn\de\db\vp$ is uniformly equivalent to $\omega$.   Moreover, generalizing Yau's third order estimate \cite{Y}, Cherrier shows that given this one can then bound $\omega+\mn\de\db\vp$ in $C^1$. Higher order estimates then follow from standard elliptic theory. A similar second order estimate was also proved by Guan-Li \cite{GL} and Zhang \cite{Zh} for general $\omega$, and sharpened in \cite{TW} in the cases of $n=2$ or $\omega$ balanced. It is also possible to avoid the third order estimate by using the Evans-Krylov theory, as in \cite{GL} and \cite{TW}.

We remark that our $L^{\infty}$ bound on $\varphi$ depends only on $(M, \omega)$ and $\sup_M F$, as in Yau's estimate for the K\"ahler case \cite{Y}.  In particular, the $L^{\infty}$ bound does not depend on $\inf_M F$.  
In the course of the proof, we say that a constant is \emph{uniform} if it depends only on the data $(M, \omega)$ and $\sup_M F$. We will often write such a constant as $C$, which may differ from line to line.  If we say that a constant depends only on a quantity $Q$ then we mean that it depends only on $Q$, $(M, \omega)$ and $\sup_M F$.

Our goal is thus to give a uniform bound for $\vp$. We begin with a lemma which can be found in  \cite{ChP}.  For the convenience of the reader, we provide a proof.  We use the notation of exterior products instead of the multilinear algebra calculations of \cite{ChP}.

\begin{lemma} \label{l1}
There are uniform constants $C, p_0$ such that for all $p\geq p_0$ we have
$$\int_M |\de e^{-\frac{p}{2}\vp}|_g^2\omega^n\leq Cp\int_M e^{-p\vp}\omega^n.$$
\end{lemma}
\begin{proof}
From now on we will use the shorthand $\omega_\vp=\omega+\mn\de\db\vp$.
Let $\alpha$ be the $(n-1,n-1)$-form given by
$$\alpha=\sum_{k=0}^{n-1}\omega_\vp^k\wedge\omega^{n-k-1}.$$
We compute, using the equation (\ref{ma}) and integrating by parts,
\begin{eqnarray} \nonumber
C\int_M e^{-p\vp}\omega^n&\geq & \int_Me^{-p\vp}(\omega_\vp^n-\omega^n) \\ \nonumber
& = & \int_M e^{-p\vp}\mn\de\db\vp\wedge\alpha\\ 
&= & p\int_M e^{-p\vp}\mn\de\vp\wedge\db\vp\wedge\alpha+\int_M e^{-p\vp}\mn \, \db \vp\wedge\de\alpha. \qquad \label{1}
\end{eqnarray}
The first term on the right hand side of (\ref{1}) is positive, and we are going to use part of it to deal with the second one. Notice that
$$\de\alpha=n\sum_{k=0}^{n-2}\omega_\vp^k\wedge\omega^{n-k-2}\wedge\de\omega.$$
Since $\de\omega$ is a fixed tensor, there is a constant $C$ so that for any $\ve>0$ and any $k$ we have the following elementary pointwise inequality
\begin{eqnarray} \nonumber
\lefteqn{ \left|\frac{\mn\, \db \vp\wedge\de\omega\wedge \omega_\vp^k\wedge\omega^{n-k-2}}{\omega^n}\right| \,  } \\  && \quad \quad \label{useful} \leq
\frac{C}{\ve}\frac{\mn\de\vp\wedge\db\vp\wedge \omega_\vp^k\wedge\omega^{n-k-1}}{\omega^n}
+\ve C \frac{\omega_\vp^k\wedge\omega^{n-k}}{\omega^n},
\end{eqnarray}
that the reader can verify by choosing local coordinates at a point that make $\omega$ the identity and $\omega_\vp$ diagonal.
Applying (\ref{useful}) we have for any $\ve>0$ and any $p$,
\begin{eqnarray*}
-\int_M e^{-p\vp} \sqrt{-1}\, \ov{\partial} \vp \wedge \partial \alpha & = &- n  \sum_{k=0}^{n-2} \int_M e^{-p\vp} \sqrt{-1} \, \ov{\partial} \vp \wedge \omega_{\vp}^k \wedge \omega^{n-k-2} \wedge \partial \omega \\
& \le & \frac{C}{\ve} \sum_{k=0}^{n-2} \int_M e^{-p\vp} \sqrt{-1} \partial \vp \wedge \ov{\partial} \vp \wedge \omega_{\vp}^k \wedge \omega^{n-k-1} \\ && \mbox{} + \ve C \sum_{k=0}^{n-2} \int_M e^{-p\vp} \omega_{\vp}^k \wedge \omega^{n-k}.
\end{eqnarray*}
Now if we choose $p_0/2 \ge C/\ve$ we see that if $0 < \ve\le 1$ then for $p \ge p_0$,
\begin{eqnarray*}
-\int_M e^{-p\vp} \sqrt{-1} \, \ov{\partial} \vp \wedge \partial \alpha 
& \le & \frac{p}{2} \int_M e^{-p\vp} \sqrt{-1} \partial \vp \wedge \ov{\partial} \vp \wedge \alpha + C \int_M e^{-p\varphi} \omega^n \\
&& \mbox{} + \ve C \sum_{k=1}^{n-2} \int_M e^{-p\vp} \omega_{\vp}^k \wedge \omega^{n-k}.
\end{eqnarray*}
Combining this with (\ref{1}) we see that for any $0<\ve <1$ there exists $p_0$ depending only on $\ve$ such that for $p\ge p_0$,
\begin{equation} \label{i2}
\frac{p}{2} \int_M e^{-p\vp} \sqrt{-1} \partial \vp \wedge \ov{\partial} \varphi \wedge \alpha \le C \int_M e^{-p\varphi} \omega^n + \ve C \sum_{k=1}^{n-2} \int_M e^{-p\vp} \omega_{\vp}^k \wedge \omega^{n-k}.
\end{equation}

We now claim the following.  There exist uniform constants $C_2, \ldots, C_n$ and $\varepsilon_0$ such that for all $\ve$ with $0< \ve \le \ve_0$, there exists a constant $p_0$ depending only on $\ve$ such that for all $p \ge p_0$ we have for $i=2, \ldots, n$,
\begin{equation} \label{induct}
\frac{p}{2^{i-1}} \int_M e^{-p\varphi} \sqrt{-1} \partial \vp \wedge \ov{\partial} \vp \wedge \alpha \le C_i \int_M e^{-p\vp} \omega^n + \ve C_i \sum_{k=1}^{n-i} \int_M e^{-p\vp} \omega_{\vp}^k \wedge \omega^{n-k}.
\end{equation}
Given the claim, the lemma follows.  Indeed once we have the statement with $i=n$ then, fixing $\ve=\ve_0$ we have for $p \ge p_0$,
\begin{eqnarray*}
\int_M |\partial e^{-\frac{p}{2} \varphi} |^2_g \omega^n & = & \frac{np^2}{4} \int_M e^{-p \vp} \sqrt{-1} \partial \vp \wedge \ov{\partial} \vp \wedge \omega^{n-1} \\
& \le & \frac{np^2}{4} \int_M e^{-p \vp} \sqrt{-1} \partial \vp \wedge \ov{\partial} \vp \wedge \alpha \\
& \le & n2^{n-3} C_n p \int_M e^{-p\vp} \omega^n,
\end{eqnarray*}
as required.

We will prove the claim by induction on $i$.  By (\ref{i2}) we have already proved the statement for $i=2$.  
So we assume the induction statement (\ref{induct}) for $i$, and prove it for $i+1$.
We compute
\begin{equation}\begin{split} \label{a1a2}
\ve C_i\sum_{k=1}^{n-i}\int_M e^{-p\vp}\omega_\vp^k\wedge\omega^{n-k} & =  \ve C_i\sum_{k=1}^{n-i}\int_M e^{-p\vp}\omega_\vp^{k-1}\wedge\omega^{n-k+1}\\
& \ \ \, + \ve C_i\sum_{k=1}^{n-i}\int_M e^{-p\vp}\mn\de\db\vp\wedge\omega_\vp^{k-1}\wedge\omega^{n-k}\\
 & = A_1 + A_2,
\end{split}\end{equation}
where 
\begin{equation*}
A_1 = \ve C_i\sum_{k=0}^{n-i-1}\int_M e^{-p\vp}\omega_\vp^{k}\wedge\omega^{n-k}, \quad A_2 = \ve C_i\sum_{k=0}^{n-i-1}\int_M e^{-p\vp}\mn\de\db\vp\wedge\omega_\vp^{k}\wedge\omega^{n-k-1}.
\end{equation*}

The term $A_1$ is already acceptable for the induction. For $A_2$ we integrate by parts to obtain
\begin{equation}\begin{split} \label{b123}
A_2 &= \ve C_ip\sum_{k=0}^{n-i-1}\int_M e^{-p\vp}\mn\de\vp\wedge\db\vp\wedge\omega_\vp^{k}\wedge\omega^{n-k-1}\\
& \ \ \, +\ve  C_i\sum_{k=1}^{n-i-1} k \int_M e^{-p\vp} \mn\, \db\vp\wedge \omega_\vp^{k-1}\wedge\omega^{n-k-1}\wedge\de\omega \\
& \ \ \, +\ve  C_i\sum_{k=0}^{n-i-1}(n-k-1) \int_M e^{-p\vp} \mn\, \db\vp\wedge \omega_\vp^{k}\wedge\omega^{n-k-2}\wedge\de\omega \\
 &=  B_1 + B_2 + B_3,
\end{split}\end{equation}
where
\begin{equation*} \begin{split}
& B_1 = \ve C_ip\sum_{k=0}^{n-i-1}\int_M e^{-p\vp}\mn\de\vp\wedge\db\vp\wedge\omega_\vp^{k}\wedge\omega^{n-k-1} \\
& B_2 = \ve  C_i\sum_{k=0}^{n-i-2} (k+1) \int_M e^{-p\vp} \mn\, \db\vp\wedge \omega_\vp^{k}\wedge\omega^{n-k-2}\wedge\de\omega \\
& B_3 = \ve  C_i\sum_{k=0}^{n-i-1}(n-k-1) \int_M e^{-p\vp} \mn\, \db\vp\wedge \omega_\vp^{k}\wedge\omega^{n-k-2}\wedge\de\omega.
\end{split}\end{equation*}
Choosing $\varepsilon_0$ such that  $\varepsilon_0 C_i<  2^{-i-1}$ we have for $\ve< \ve_0$ and $p \ge p_0$,
\begin{equation} \label{b1}
B_1 \le \frac{p}{2^{i+1}} \int_M e^{-p\varphi} \sqrt{-1} \partial \varphi \wedge \ov{\partial} \vp \wedge \alpha.
\end{equation}
For the terms $B_2$ and $B_3$ we use again \eqref{useful} to obtain
\begin{equation} \begin{split} \label{b23}
B_2 + B_3 & \le nC_i C\sum_{k=0}^{n-i-1}\int_M e^{-p\vp} \mn\de\vp\wedge\db\vp\wedge \omega_\vp^k\wedge\omega^{n-k-1} \\
& \ \ \, + \ve^2 n C_i C \sum_{k=0}^{n-i-1}\int_M e^{-p\vp} \omega_\vp^k\wedge\omega^{n-k}.
\end{split} \end{equation}
Notice that the second term on the right hand side of (\ref{b23}) is acceptable for the induction.  Moreover, we may assume that $p_0 \ge 2^{i+1} nC_i C $ and thus for $p \ge p_0$,
\begin{equation} \begin{split} \label{b232}
B_2 + B_3 & \le \frac{p}{2^{i+1}} \int_M e^{-p\varphi} \sqrt{-1} \partial \varphi \wedge \ov{\partial} \vp \wedge \alpha \\
& \ \ \, + \ve^2 n C_i C \sum_{k=0}^{n-i-1}\int_M e^{-p\vp} \omega_\vp^k\wedge\omega^{n-k}.
\end{split} \end{equation}

Combining the inductive hypothesis (\ref{induct}) with (\ref{a1a2}), (\ref{b123}), (\ref{b1}), (\ref{b232}) we obtain for $p \ge p_0$,
\begin{equation} \label{induct2}
\frac{p}{2^{i}} \int_M e^{-p\varphi} \sqrt{-1} \partial \vp \wedge \ov{\partial} \vp \wedge \alpha \le C_{i+1} \int_M e^{-p\vp} \omega^n + \ve C_{i+1} \sum_{k=1}^{n-i-1} \int_M e^{-p\vp} \omega_{\vp}^k \wedge \omega^{n-k},
\end{equation}
completing the inductive step.  This finishes the proof of the claim and thus the lemma.
\end{proof}

We now complete the proof of the Main Theorem.
Using Lemma \ref{l1} and the Sobolev inequality, we have for $\beta=\frac{n}{n-1}>1$,
\begin{equation*}\begin{split}
\left(\int_M e^{-p\beta \vp}\omega^n\right)^{1/\beta}&\leq C\left(
\int_M|\de e^{-\frac{p}{2}\vp}|^2\omega^n+ \int_M e^{-p\vp}\omega^n\right)\\
&\leq Cp\int_M e^{-p\vp}\omega^n,
\end{split}\end{equation*}
for all $p \ge p_0$.  Thus
$$\|e^{-\vp}\|_{L^{p\beta}}\leq C^{1/p}p^{1/p}\|e^{-\vp}\|_{L^p}.$$
Since this holds for all $p\geq p_0$, we can iterate this estimate in a standard way to obtain
$$\|e^{-\vp}\|_{L^{\infty}}\leq C \|e^{-\vp}\|_{L^{p_0}},$$
which is equivalent to
$$e^{-p_0\inf_M\vp}\leq C\int_M e^{-p_0\vp}\omega^n.$$

We now make use of a result from \cite{TW}:

\begin{lemma}  Let $f$ be a smooth function on $(M, \omega)$.  Write $d\mu = \omega^n/\int_M \omega^n$.  If there exists a constant $C_1$ such that
\begin{equation}
e^{-\inf_M f} \le e^{C_1} \int_M e^{-f} d\mu,
\end{equation}
then
\begin{equation}
| \{ f \le \inf_M f + C_1 +1 \} | \ge \frac{e^{-C_1}}{4},
\end{equation}
where $| \,  \cdot \, |$ denotes the volume of the set with respect to $d \mu$.
\end{lemma}
\begin{proof} See  \cite[Lemma 3.2]{TW}.
\end{proof}

Applying this lemma to $f=p_0 \varphi$ we see that there exist uniform constants $C,\delta>0$ so that
\begin{equation} \label{set}
|\{\vp\leq \inf_M\vp+C\}|\geq\delta.
\end{equation}
We remark that, in \cite{TW}, the bound (\ref{set}) is established whenever one has the improved second order estimate,
\begin{equation} \label{improve}
\tr{\omega}{\omega_{\varphi}} \le C e^{A(\varphi- \inf_M \varphi)},
\end{equation}
for uniform $A$ and $C$.  It is shown in \cite{TW} that (\ref{improve}) holds if $n=2$ or $\omega$ is balanced.

The $L^{\infty}$ bound on $\varphi$, and hence the Main Theorem, now follow from the arguments of \cite{TW}.   However, we include an outline of these arguments for the reader's convenience.
Recall that, from \cite{Ga}, if $(M, \omega)$ is a compact Hermitian manifold then there exists a unique smooth function $u: M \rightarrow \mathbb{R}$ with $\sup_M u=0$ such that the metric 
$\omega_{\textrm{G}} = e^u \omega$ is \emph{Gauduchon}, that is, satisfies
\begin{equation} \label{gaud}
\partial \ov{\partial} (\omega_{\textrm{G}}^{n-1})=0.
\end{equation}
Writing $\Delta_{\textrm{G}}$ for the complex Laplacian associated to $\omega_{\textrm{G}}$ (which differs from the Levi-Civita Laplacian in general),
we have the following lemma.

\begin{lemma} \label{lemmapsi} Let $M$ be a compact complex manifold of complex dimension $n$ with a Gauduchon metric $\omega_{\emph{G}}$.  If $\psi$ is a smooth nonnegative function on $M$ with
$$\Delta_{\emph{G}} \psi \ge -C_0$$ 
then there exist constants $C_1$ and $C_2$ depending only on $(M, \omega_{\emph{G}})$ and $C_0$ such that:
\begin{equation} \label{psi1}
\int_M |\partial \psi^{\frac{p+1}{2}}|_{\omega_{\emph{G}}}^2 \omega_{\emph{G}}^n \le C_1 p \int_M \psi^p \omega_{\emph{G}}^n \quad \textrm{for all } p \ge 1,
\end{equation}
and
\begin{equation} \label{psi2}
\sup_M \psi \le C_2  \max \left\{  \int_M \psi \, \omega_{\emph{G}}^n, 1 \right\}.
\end{equation}
\end{lemma}
\begin{proof} Although \cite[Lemma 3.4]{TW} is stated for complex dimension $2$, the same proof works for any dimension.
\end{proof}

We apply Lemma \ref{lemmapsi} to the function
$\psi=\vp-\inf_M\vp$,
which satisfies
$\Delta_{\textrm{G}} \psi = e^{-u} \Delta \psi > -C$, where $\Delta$ is the complex Laplacian with respect to $\omega$.
In light of \eqref{psi2}, once we bound the $L^1$ norm of $\psi$ the Main Theorem follows.
Denoting by $\underline{\psi}$ the average of $\psi$ with respect to
$\omega_{\textrm{G}}^n$ we obtain from the Poincar\'e inequality and (\ref{psi1}) with $p=1$,
\begin{equation} \label{poincare}
\|\psi-\underline{\psi}\|_{L^2}\leq C \left( \int_M |\partial \psi |_{\omega_{\textrm{G}}}^2 \omega_{\textrm{G}}^n \right)^{1/2}  \leq C\|\psi\|^{1/2}_{L^1}.
\end{equation}
In (\ref{poincare}) and the following we are using $L^q$ norms with respect to the volume form $\omega_{\textrm{G}}^n$, which are equivalent to $L^q$ norms with respect to $d\mu$.
Using (\ref{set}) we see that the set 
 $S:=\{  \psi \leq C \}$ satisfies
$|S|_{\textrm{G}} \ge \delta$
 for a uniform $\delta>0$, where $| \cdot |_{\textrm{G}}$ denotes the volume of a set with respect to $\omega_{\textrm{G}}^n$. 
Hence
$$\frac{\delta}{\int_M \omega_{\textrm{G}}^n} \int_M\psi \omega_{\textrm{G}}^n =\delta \underline{\psi}\leq  \int_S \underline{\psi} \omega_{\textrm{G}}^n  \leq 
\int_S (|\psi-\underline{\psi}|+C)\omega_{\textrm{G}}^n \leq  \int_M|\psi-\underline{\psi}| \omega_{\textrm{G}}^n +C.$$
Then, 
\begin{eqnarray*}
\|\psi \|_{L^1}
 & \leq & C( \|\psi-\underline{\psi}\|_{L^1}+1)  \leq    C( \|\psi-\underline{\psi}\|_{L^2}+1)  \le  C ( \| \psi \|_{L^1}^{1/2} +1),
\end{eqnarray*}
which shows that $\psi$ is uniformly bounded in $L^1$.  This  completes the proof of the Main Theorem.

Finally we mention that corollary 1 follows from the argument of Cherrier \cite{ChP}, which uses results from \cite{De}.  Or for another proof, see    \cite[Corollary 1]{TW}.

\bigskip
\noindent
{\bf Acknowledgements} \ 
The authors  thank S.-T. Yau for many useful discussions over the last few years on the complex Monge-Amp\`ere equation. The authors also express their gratitude to D.H. Phong for his support, encouragement and helpful suggestions.  In addition the authors thank G. Sz\'ekelyhidi for some helpful discussions.

\bigskip
\noindent
Mathematics Department, Columbia University, 2990 Broadway, New York, NY 10027

\bigskip
\noindent
Mathematics Department, University of California, San Diego, 9500 Gilman Drive \#0112, La Jolla CA 92093

\end{document}